\documentclass[a4paper, 12pt, reqno]{amsart}
\usepackage{amsmath}
\usepackage{amssymb}
\usepackage{graphics}
\usepackage{mathrsfs}
\usepackage[pdflatex]{graphicx}
\usepackage[marginratio=1:1,tmargin=117pt, height=650pt]{geometry}
\usepackage{color}
\usepackage{amsthm}

\usepackage[all]{xy}

\usepackage{eurosym}

\usepackage[english]{babel}
\usepackage[T1]{fontenc}
\usepackage[latin1]{inputenc}
\usepackage{marginnote}

\usepackage{amsmath,amsfonts,amsthm,amssymb,verbatim,enumerate,quotes,graphicx}
\usepackage[flushmargin]{footmisc}
\usepackage[noadjust]{cite}
\usepackage{color}
\newcommand{\comments}[1]{}
\numberwithin{equation}{section}
\makeatletter
\def\blfootnote{\xdef\@thefnmark{}\@footnotetext}
\makeatother

\definecolor{orange}{rgb}{1,0.5,0}

\theoremstyle{plain}
\newtheorem{theorem}{Theorem}[section]

\newtheorem{lemma}[theorem]{Lemma}

\theoremstyle{definition}

\newtheorem{ex}{Example}[section]

\theoremstyle{remark} 

\newtheorem{remark}{Remark}[section]

\begin{document}
\title[Regularity and growth conditions for fast escaping points]{Regularity and growth conditions for fast escaping points of entire functions}
\author{V. Evdoridou}
\address{Department of Mathematics and Statistics\\ The Open University\\ Walton Hall\\ Milton Keynes MK7 6AA\\ United Kingdom}
\email{vasiliki.evdoridou@open.ac.uk}
\date{\today}
\begin{abstract}
Let $f$ be a transcendental entire function. The quite fast escaping set, $Q(f)$, and the set $Q_2(f),$ which was defined recently in \cite{Evd16}, are equal to the fast escaping set, $A(f),$ under certain conditions. In this paper we generalise these sets by introducing a family of sets $Q_m(f)$, $m \in \mathbb{N}.$ We also give one regularity and one growth condition which imply that $Q_m(f)$ is equal to $A(f)$ and we show that all functions of finite order and positive lower order satisfy $Q_m(f)=A(f)$ for any $m$. Finally, we relate the new regularity condition to a sufficient condition for $Q_2(f)=A(f)$ introduced in \cite{Evd16}. 
\end{abstract}

\maketitle
\section{Introduction}

 Let $f$ be a transcendental entire function. The set of points $z \in \mathbb{C}$ for which $(f^n)_{n \in \mathbb{N}}$ forms a normal family in some neighbourhood of $z$ is called the \textit{Fatou set} $F(f)$ and the complement of $F(f)$ is the \textit{Julia set} $J(f)$. An introduction to the properties of these sets can be found in \cite{Berg}.
 
A lot of work has been done in recent years on a conjecture of Eremenko on the escaping set of $f$. The \textit{escaping set} $I(f)$ of $f$ is defined as follows: 
 $$I(f)= \{z \in \mathbb{C}:f^n(z) \to \infty\}$$ and it was first studied by Eremenko in \cite{Ere} who showed that for any transcendental entire function $f,$ we have $I(f) \cap J(f) \neq \emptyset$, 
 $J(f)= \partial I(f)$ and
 all the components of $\overline{I(f)}$ are unbounded.

 His conjecture, that all the components of $I(f)$ are unbounded, is still an open question. Significant progress has been made on the conjecture by Rippon and Stallard who proved that $I(f)$ has at least one unbounded component (see \cite[Theorem 1]{F-E}). In order to do this, they considered  a subset of the escaping set known as the \textit{fast escaping set}, $A(f)$. This set was introduced by Bergweiler and Hinkannen in \cite{B-H}. We will use the definition given by Rippon and Stallard in \cite{Fast} according to which
$$A(f)= \{z: \;\text{there exists}\;\ell \in \mathbb{N}\;\text{such that}\; \lvert f^{n+\ell}(z) \rvert \geq M^n(R,f),\;\text{for}\;n \in \mathbb{N}\},$$
where $$M(r,f)= M(r) = \max_{\lvert z\rvert =r} \lvert f(z)\rvert, \;\;\text{for}\;\;r>0,$$ and $R>0$ is large enough to ensure that $M(r) >r$ for $r \geq R.$ In the same paper they showed that $A(f)$ has  properties similar to the properties of $I(f)$ listed above. 
(Some of these results were shown in \cite{B-H}.)

The set $A(f)$ also has other nice properties (described in \cite{Fast}) and plays a key role in iteration of transcendental entire functions and so it is useful to be able to identify points that are fast escaping. In \cite[Theorem 2.7]{Fast}, it is shown that  points which eventually escape faster than the iterates of the function $\mu_{\varepsilon}$ defined by $\mu_{\varepsilon}(r)= \varepsilon M(r), \varepsilon \in (0,1), r>0$,  are actually fast escaping.

It is natural to ask whether this $\mu_{\varepsilon}$ can be replaced by a smaller function. In this context,
 Rippon and Stallard introduced the \textit{quite fast escaping set} $Q(f)$ in \cite{Regul} and in \cite{Evd16} we generalised this and introduced the following family of sets:
 
$$Q_m(f)= \{z: \exists \;\;\varepsilon \in (0,1), \ell \in \mathbb{N}\;\; \text{such that}\;\; \lvert f^{n+ \ell}(z) \rvert \geq \mu_{m,\varepsilon}^n(R),\; \text{for}\;\; n\in \mathbb{N}\},$$
where $\mu_{m,\varepsilon}$ is defined by
$$\log^m \mu_{m,\varepsilon}(r)= \varepsilon \log^m M(r),\;\;m \in \mathbb{N},\;\;\varepsilon \in (0,1).$$
In this paper, we only consider the case when there exists
 $R>0$ such that $\mu_{m,\varepsilon}(r)> r$ for $r \geq R$; in particular, we always have $Q_m(f)\subset I(f).$  For $m=1$ we obtain the quite fast escaping set $Q(f)$; that is, $Q_1(f)=Q(f).$

 Note that for $0<\varepsilon<1$ we have $\mu_{m,\varepsilon}(r)< \mu_{1,\varepsilon}(r)< M(r),$ for any $m\geq 2$ and for $r$ large enough, so
 $$A(f)\subset Q(f) \subset Q_m(f)\subset I(f).$$
In Section 2 we give a large class of functions for which  $\mu_{m,\varepsilon}(r)$ is greater than $r$ for $r$ large enough.

In \cite{Evd16} we considered the case $m=2$, that is, 
 \begin{equation}
 \label{defmu}
\mu_{2,\varepsilon}(r)= \exp((\log M(r))^{\varepsilon})
\end{equation}
 and we found regularity conditions which imply that $Q_2(f)=A(f)$. In particular, we proved that any transcendental entire function of finite order and positive lower order satisfies $Q_2(f)=A(f)$. 
 
 In this paper we introduce new techniques which enable us to generalise the result for any $m \in \mathbb{N}$ as given in the following theorem:
 
 \begin{theorem}
\label{ordqm}
Let $f$ be a transcendental entire function of finite order and positive lower order. Then $Q_m(f)=A(f), m \in \mathbb{N}.$
\end{theorem}

In particular, Theorem \ref{ordqm} implies that functions in the Eremenko-Lyubich class $\mathcal{B}$ which have finite order satisfy $Q_m(f)=~A(f).$ Indeed, functions in the class $\mathcal{B}$ have positive lower order and in fact have lower order not less than $1/2$  (see \cite[Lemma 3.5]{Dimen}). Note that the class $\mathcal{B}$ consists of transcendental
entire functions whose set of singular values (that is, critical values and asymptotic
values) is bounded, and it is much studied in complex dynamics. Classes of functions that satisfy the hypothesis of Theorem \ref{ordqm} were studied, for example, in \cite{Haus} and \cite{3R}.

We prove the theorem in two different ways. The first proof is based on a new regularity condition and the second on a growth condition.
In Section 3 we give our first proof of Theorem \ref{ordqm}  which is in  two steps. We first introduce a new regularity condition which we call  $m$-log-regularity and which implies that $Q_m(f)=A(f).$ Let $f$ be a transcendental entire function. Then $f$ is \textit{$m$-log-regular} if and only if, for any $\varepsilon \in (0,1),$ there exist $R>0$ and $k>1$ such that 
\begin{equation}
\label{mlog}
\mu_{m,\varepsilon}(\exp^{m-1}(r^k)) \geq \exp^{m-1}(M(r)^k),\;\;\text{for}\;\;r \geq R.
\end{equation}
For $m=1$ we obtain the log-regularity condition which was first introduced by Anderson and Hinkkanen in \cite{A-H} and was used by Rippon and Stallard in \cite{Regul} as a sufficient condition for $Q(f)$ to be equal to $A(f).$ We then show that any function which is $m$-log-regular satisfies $Q_m(f)=A(f).$ In the second step, we prove that all functions of finite order and positive lower order are $m$-log-regular.

In Section 4 we prove Theorem \ref{ordqm} in a different way, again in two steps. We give a growth condition which is sufficient for $Q_m(f)=A(f)$ and then we show that any transcendental entire function of finite order and positive lower order satisfies this growth condition.

In \cite{Evd16} we introduced a regularity condition called \textit{strong log-regularity} which implies that $Q_2(f)=A(f).$ In Section 5 we show how strong log-regularity is related to $2$-log-regularity. In particular we prove that a strongly log-regular function of finite order is always $2$-log regular and we give an example of a $2$-log-regular function of finite order which fails to be strongly log-regular.

\textit{Acknowledgments.} I would like to thank my supervisors Prof. Phil Rippon and Prof. Gwyneth Stallard for their help in the preparation of this paper.

\section{Properties of $Q_m(f)$}
In this section we prove some basic properties of $Q_m(f).$
Just as for $\mu_{2,\varepsilon}$, in the general case we do not know a priori that, for any given transcendental entire function, $\mu_{m,\varepsilon}(r)$ is greater than $r$ for $r$ large enough.  We show first that, for a large class of functions, there is always a positive $R$ such that $\mu_{m,\varepsilon}(r) > r$, for $r \geq R,$ and hence for these functions $Q_m(f)$ is defined.
\begin{theorem}
\label{mu>r}
Let $f$ be a transcendental entire function, $m \geq 2$ and $\varepsilon \in  (0,1)$. If there exist $q>0$, $r_0>0$ and $n \in \mathbb{N}$ such that 
\begin{equation}
\label{mur}
M(r) \geq \exp^{n+1}((\log^n r)^q),\;\;\text{for}\;\;r \geq r_0,
\end{equation}
 then, for any $c>1$, there exists $R>0$ such that
$$\mu_{m,\varepsilon}(r)>cr,\;\;\text{for}\;\;r \geq R.$$

\end{theorem}
Note that (\ref{mur}) is true for all functions of positive lower order as well as some functions of zero lower order. In particular, it is true for all the functions in class $\mathcal{B}$ as they have lower order not less than $1/2$.
In order to prove Theorem  \ref{mu>r} we use the following inequality.
\begin{lemma}
\label{indu}
For any $n \in \mathbb{N}, p\geq 1$ and $a_1,...,a_n, b_1,...,b_n >0$ there exists $R>0$ such that 
\begin{equation}
\label{lo2}
a_1\log (a_2\log...\log(a_n r)...) \geq \log (b_1 \log...\log((b_n r)^p)...),\;\;\text{for}\;\;r\geq R.
\end{equation}
\end{lemma}
\begin{proof}
It suffices to prove (\ref{lo2}) for $p>1.$ We use proof by induction. As $a_1r \geq \log ((b_1r)^p)$, for $r$ large enough, (\ref{lo2}) is certainly true for $n= 1.$ 
Suppose now that (\ref{lo2}) is true for some $n \geq 2.$ We will deduce that
\begin{equation}
\label{lo4}
a_1 \log (a_2 \log... \log (a_n \log (a_{n+1}r))...)\geq \log(b_1 \log...\log(b_n\log((b_{n+1}r)^p))...),
\end{equation} 
for $r$ large enough. To do this, note first that, for $r$ large enough,
$$a_1 \log (a_2 \log... \log (a_n \log (a_{n+1}r))...)\geq \log(b_1 \log...\log((b_n\log (a_{n+1}r))^p)...),$$
by (\ref{lo2}). Then, in order to deduce (\ref{lo4}) it suffices to show that 
\begin{equation}
\label{lo5}
(b_n \log (a_{n+1}r))^p \geq b_n p \log (b_{n+1}r),
\end{equation}
for $r$ large enough.
Note now that (\ref{lo5}) is true since there exists $R (=R(n))>0$ such that
$$\frac{b_n^{p-1}}{p}(\log (a_{n+1}r))^p \geq \log (b_{n+1}r),\;\;\text{for}\;\;r\geq R,$$
and the result follows.\end{proof}

\begin{proof}[Proof of Theorem \ref{mu>r}]

By definition, $\mu_{m, \varepsilon}(r)= \exp^m( \varepsilon \log^mM(r)),$ so we have to show that
\begin{equation}
\label{mur1}
\exp^m( \varepsilon \log^mM(r))>cr,\;\;\text{for}\;\;r\;\;\text{large enough}.
\end{equation}
We consider three different cases depending on the relative sizes of $m$ and the positive integer $n$ from (\ref{mur}).

a) Suppose that $n+1=m.$ Then, by (\ref{mur}),
\begin{eqnarray}
 \exp^m( \varepsilon \log^mM(r)) &\geq& \exp^m( \varepsilon \log^m(\exp^{n+1}((\log^n r)^q)))\nonumber \\
 &=& \exp ^m (\varepsilon(\log^nr)^q)\nonumber \\
 & >&cr,\;\;\text{for}\;\;r\;\;\text{large enough},
 \end{eqnarray}
since $\varepsilon(\log^nr)^q > \log \log^n cr= \log^m cr,$ for $r$ large enough.

b) Suppose that $n+1<m.$ Then, by (\ref{mur}), 
$$ \exp^m( \varepsilon \log^mM(r)) \geq \exp^m( \varepsilon \log^m(\exp^{n+1}((\log^n r)^q)))= \exp ^m( \varepsilon \log^{m-n-1}((\log^nr)^q)).$$
Hence, we need to show that, for any $c>1$, 
$$\exp ^m( \varepsilon \log^{m-n-1}((\log^nr)^q))>cr,\;\;\text{for}\;\;r\;\;\text{large enough},$$
or, equivalently,
\begin{equation}
\label{mur2}
\varepsilon \log^{m-n-1}((\log^nr)^q)>\log^m cr,\;\;\text{for}\;\;r\;\;\text{large enough},
\end{equation}
which holds by applying  Lemma \ref{indu} with $n$ replaced by $m$, $p=1,$ $a_1= \varepsilon, a_{m-n-1}=q, b_{m}=c$ and all the other coefficients equal to $1$.

c) Finally, suppose that $n+1>m.$ Then, by (\ref{mur}),
$$ \exp^m( \varepsilon \log^mM(r)) \geq \exp^m( \varepsilon \log^m(\exp^{n+1}((\log^n r)^q)))= \exp ^m( \varepsilon\exp^{n+1-m}(( \log^nr)^q)).$$
Hence, we need to show that, for any $c>1$,
$$\exp ^m( \varepsilon \exp^{n+1-m}((\log^nr)^q))>cr,\;\;\text{for}\;\;r\;\;\text{large enough},$$
or, equivalently,
\begin{equation}
\label{mur3}
\varepsilon \exp^{n+1-m}((\log^nr)^q)>\log^m cr,\;\;\text{for}\;\;r\;\;\text{large enough}.
\end{equation}
Note now that (\ref{mur3}) is equivalent to 
\begin{equation}
\label{mur4}
(\log^nr)^q>\log^{n+1-m}\left(\frac{1}{\varepsilon}\log^m cr \right),\;\;\text{for}\;\;r\;\;\text{large enough}.
\end{equation}
If we apply Lemma \ref{indu} with $n$ replaced by $n+2$, $p=1$, $a_1=q$, $b_{n+2-m}=1/\varepsilon, b_{n+2}=c$ and the rest of the coefficients equal to $1$ we obtain
$$q\log^{n+1}r>\log^{n+2-m}\left(\frac{1}{\varepsilon}\log^m cr \right),$$
for $r$ large enough and so (\ref{mur4}) follows.\end{proof}

We now show that $Q_m(f)$ has some basic properties similar to those  of $I(f), A(f)$ and $Q(f).$

\begin{theorem}
Let $f$ be a transcendental entire function and $m \in \mathbb{N}$. Then
$$Q_m(f) \neq \emptyset,\;\;Q_m(f) \cap J(f) \neq \emptyset,\;\;\text{and}\;\;J(f)= \overline{Q_m(f) \cap J(f)}.$$ If, in addition, for any $c>1$, there exists $R>0$ such that 
\begin{equation}
\label{mucr}
\mu_{m, \varepsilon}(r) > cr,\;\;\text{for}\;\;r \geq R,
\end{equation}
 then $$\;J(f)= \partial Q_m(f),$$
and $\overline{Q_m(f)}$ has no bounded components.
\end{theorem}
\begin{proof}
All the properties above hold for $A(f)$ (see \cite{Fast}). As $A(f) \subset Q_m(f)$, we certainly have $Q_m(f) \neq \emptyset$ and $Q_m(f) \cap J(f) \neq \emptyset.$ Also $J(f)= \overline{A(f) \cap J(f)} \subset \overline{Q_m(f) \cap J(f)}.$ Since $J(f)$ is closed, we also have $\overline{Q_m(f) \cap J(f)} \subset J(f)$ and so the third property is also true.

In order to prove the two remaining properties, we follow the arguments in the proof of \cite[Theorem 2.1]{Regul}.
Note first that $Q_m(f)$ is infinite and completely invariant under $f$ which, since $J(f)$ is the smallest closed completely invariant set with at least three points, implies that $J(f) \subset \overline{Q_m(f)}$. But any open subset of $Q_m(f)$ is contained in $F(f)$ since it contains no periodic points of $f$, and so $J(f) \subset \partial Q_m(f).$ 

Suppose now that $\partial Q_m(f) \cap U \neq \emptyset,$ where $U$ is a  Fatou component. Then $Q_m(f) \cap U \neq \emptyset,$ and we take $z \in Q_m(f) \cap U$. Then there will be a disc $\Delta$ such that $z \in \Delta$ and $\overline{\Delta} \subset U$. If $U$ is simply connected then, by applying \cite[Lemma 7]{Berg}, we have that there exists $C>0$ such that
$$
\lvert f^n(z')\rvert \geq C\lvert f^n(z)\rvert,$$
for any $z' \in \Delta$ and $n \in \mathbb{N}$.
Hence, by (\ref{mucr}), there exists $R>0$ such that $\mu_{m,\varepsilon}(r) >r$, $r\geq R$ and $\ell \in \mathbb{N}$ such that, for $n \in \mathbb{N},$
$$\lvert f^{n+\ell}(z')\rvert \geq C\lvert f^{n+\ell}(z)\rvert \geq C\mu_{m,\varepsilon}^n(R),$$
and so
$$\lvert f^{n+\ell}(z')\rvert \geq \mu_{m,\varepsilon}^{n-1}(R),$$ by (\ref{mucr}).
Therefore, any point in the neighbourhood $\Delta$ of $z$ lies in $Q_m(f)$ which gives a contradiction. If the Fatou component $U$ is multiply connected then $U \subset A(f) \subset Q_m(f)$ (see \cite[Theorem 2]{F-E}) and so there is again a contradiction. Hence, $\partial Q_m(f)\subset J(f).$ 

Finally, if $\overline{Q_m(f)}$ has a bounded component, $E$ say, then there is an open topological annulus $A$ lying in the complement of $\overline{Q_m(f)}$ that surrounds $E$. Since $\overline{Q_m(f)}$ is completely invariant under $f$, $A$ is contained in $F(f)$ by Montel's theorem. But from the previous property, $J(f)= \partial Q_m(f)$ and so $a$ is contained in a multiply connected Fatou component. As any multiply connected Fatou component is contained in $A(f) \subset Q_m(f)$  we deduce that $A \subset Q_m(f)$ which gives a contradiction.
\end{proof}

\section{Regularity conditions for $Q_m(f)=A(f)$}
In this section, we use regularity conditions to prove Theorem \ref{ordqm}. In the introduction we defined $m$-log-regularity which is a sufficient condition for $Q_m(f)$ to be equal to $A(f).$ In fact, there also exists another regularity condition called $m$-weak-regularity which is equivalent to $Q_m(f)=A(f).$ We will show later that $m$-log-regularity is stronger than $m$-weak-regularity and hence if $f$ is $m$-log-regular then $Q_m(f)= A(f).$ Finally, we will use these ideas in order to prove Theorem \ref{ordqm}. Note that $m$-log-regularity is easier to check than $m$-weak-regularity which is defined as follows:

Let $R>0$ be any value such that $M(r)>r$ for $r \geq R$. We say that $f$ is $m$\textit{-weakly regular} if for any $\varepsilon \in (0,1)$  there exists $r=r(R)>0$ such that 
$$\mu_{m,\varepsilon}^n(r) \geq M^n(R), \;\;\text{for}\;\;n \in \mathbb{N},$$
or, equivalently, if there exists $\ell = \ell(R) \in \mathbb{N}$ such that 
$$\mu_{m,\varepsilon}^{n+\ell}(R) \geq M^n(R), \;\;\text{for}\;\;n \in \mathbb{N}.$$
For $m=1$ we have the weak-regularity that was introduced by Rippon and Stallard in \cite{Regul}.\\

We will show that $m$-weak-regularity is a necessary and sufficient condition for $f$ to satisfy $Q_m(f)=A(f).$ In order to prove our result we make use of the following theorem of Rippon and Stallard (see \cite[Theorem 3.1]{Regul}).
\begin{theorem}
\label{RSth}
Let $f$ be a transcendental entire function. There exists $R= R(f)>0$ with the property that whenever $(a_n)$ is a positive sequence such that 
\begin{equation}
\label{sq1.7}
a_n \geq R \;\;\text{and}\;\;a_{n+1} \leq M(a_n),\;\;\text{for}\;\; n\in \mathbb{N},
\end{equation}
there exists a point $\zeta \in J(f)$ and a sequence $(n_j)$ with $n_j \to \infty$ such that 
\begin{equation}
\label{sq1.8}
\lvert f^n(\zeta) \rvert \geq a_n, \;\;\text{for}\;\;n \in \mathbb{N},\;\;\text{but}\;\;\lvert f^{n_j}(\zeta) \rvert \leq M^2(a_{n_j}),\;\;\text{for}\;\;j \in \mathbb{N}.
\end{equation}
\end{theorem}
We now prove our result.
\begin{theorem}
\label{m-weak_necsuf}
Let $f$ be a transcendental entire function. Then $f$ is $m$-weakly regular if and only if $Q_m(f)=A(f).$
\end{theorem}
\begin{proof}
Suppose that $f$ is $m$-weakly regular and let $R>0$ be such that $M(r)>r$ for $r\geq R.$ Then there exists $r=r(R)>0$ such that 
$$\mu_{m,\varepsilon}^n(r) \geq M^n(R), \;\;\text{for}\;\;n \in \mathbb{N}.$$
If $z \in Q_m(f)$, then there exist $\varepsilon \in (0,1)$ and $\ell \in \mathbb{N}$ such that 
$$\lvert f^{n+\ell}(z)\rvert \geq \mu_{m,\varepsilon}^n(R),\;\;\text{for}\;\;n \in \mathbb{N}.$$
Let $r=r(R)$ be as above. Then there exists $N \in \mathbb{N}$ such that $\mu_{m,\varepsilon}^N(R)>r$ so
$$\lvert f^{n+\ell+N}(z)\rvert \geq \mu_{m,\varepsilon}^{n+N}(R) \geq \mu_{m,\varepsilon}^n(r) \geq M^n(R),\;\;\text{for}\;\;n \in \mathbb{N},$$ and hence $z \in A(f)$. Thus $Q_m(f) \subset A(f).$ Clearly $A(f) \subset Q_m(f)$ and so we have $Q_m(f)= A(f)$ as claimed.

In order to show that the opposite direction of the theorem is also true we will prove that if $f$ is not $m$-weakly regular then $Q_m(f)\setminus A(f)$ is non-empty. Take $R>0$ such that $\mu_{m,\varepsilon}(r) \geq r$, for $r \geq R$. Since $f$ is not weakly-log-regular, for any $\ell \in \mathbb{N}$ there exists $n(\ell) \in \mathbb{N}$ such that $\mu_{m,\varepsilon}^{n(\ell) + \ell}(R) < M^{n(\ell)}(R)$ and hence,  for any $n \in \mathbb{N}$ with $n> n(\ell),$ we have 
 \begin{equation}
 \label{notmweak}
 \mu_{m,\varepsilon}^{n+\ell}(R)< M^n(R).
 \end{equation}

Now, by Theorem \ref{RSth}, with $a_n= \mu_{m,\varepsilon}^{n}(R), n \in \mathbb{N},$ there exists a point $\zeta$ and a sequence $(n_j) \to \infty$ as $j \to \infty $, such that
\begin{equation}
\label{sq1.9}
\lvert f^n(\zeta) \rvert \geq \mu_{m,\varepsilon}^{n}(R),\;\;\text{for}\;\;n \in \mathbb{N},
\end{equation}
and
\begin{equation}
\label{sq1.10}
\lvert f^{n_j}(\zeta) \rvert \leq M^2(\mu_{m,\varepsilon}^{n_j}(R)), \;\;\text{for}\;\;j \in \mathbb{N}.
\end{equation}
It follows from (\ref{sq1.9}) that $\zeta \in Q_m(f).$ Also, (\ref{notmweak}) and (\ref{sq1.10}) together imply that, for each $\ell \in \mathbb{N}$ and sufficiently large values of $j$, we have
\begin{eqnarray}
\lvert f^{(n_j-\ell +2)+\ell -2}(\zeta) \rvert &=& \lvert f^{n_j}(\zeta) \rvert \nonumber \\
& \leq & M^2(\mu_{m,\varepsilon}^{n_j}(R))\nonumber \\
&<& M^2(M^{n_j-\ell}(R))\nonumber \\
&=&  M^{n_j-\ell+2}(R).\nonumber
\end{eqnarray}
 Hence, $\zeta \notin A(f),$ so $Q_m(f) \neq A(f)$, as required.
\end{proof}

We now give the proof of Theorem \ref{ordqm}. The proof is in two steps. First, we prove the following result which implies that all $m$-log-regular functions satisfy $Q_m(f)=A(f).$

\begin{theorem}
\label{mlogmweak}
Let $f$ be a transcendental entire function. If $f$ is $m$-log-regular, then $f$ is $m$-weakly regular and hence $Q_m(f)=A(f)$.
\end{theorem}
\begin{proof}
Suppose that $f$ is $m$-log-regular and let $0< \varepsilon<1.$ Let $R>0$ be so large that $M(r)>r$ for $r \geq R.$ Since $f$ is $m$-log-regular, for any $\varepsilon \in (0,1)$ there exists $r_0\geq R$ and $k>1$ such that
$$\mu_{m,\varepsilon}(\exp^{m-1}(r^k)) \geq \exp^{m-1}(M(r)^k),\;\;\text{for}\;\;r\geq r_0.$$
Hence, 
\begin{eqnarray}
\mu_{m,\varepsilon}(\mu_{m,\varepsilon}(\exp^{m-1}(r^k))) &\geq & \mu_{m,\varepsilon}(\exp^{m-1}(M(r)^k))\nonumber \\
& \geq & \exp^{m-1}((M(M(r))^k)\nonumber 
\end{eqnarray}
and so, using this argument repeatedly, we have
$$\mu_{m,\varepsilon}^n(\exp^{m-1}(r^k)) \geq \exp^{m-1}(M^n(r)^k),\;\;\text{for}\;\;r \geq r_0\;\;\text{and}\;\;n \in \mathbb{N}.$$
Thus, whenever $r \geq r_0$, we have
$$\mu_{m,\varepsilon}^n(\exp^{m-1}(r^k)) \geq M^n(r) \geq M^n(R),\;\;\text{for}\;\;n \in \mathbb{N},$$
and so $f$ is $m$-weakly regular. Hence, by Theorem \ref{m-weak_necsuf}, $Q_m(f)=A(f).$ \end{proof}

The second part of the proof of Theorem \ref{ordqm} is to show that all functions of finite order and positive lower order are $m$-log-regular. In order to prove this we will need the following lemma.

\begin{lemma}
\label{genlemma2}
For any $n \in \mathbb{N}$ and any $d>0$, $0<q<1$, there exists $R>0$ such that 
\begin{equation}
\label{gl01}
\log^n (r^q) > d(\log^n r)^q,\;\;\text{for}\;\;r\geq R.
\end{equation}
\end{lemma}
\begin{proof}
We will prove (\ref{gl01}) using induction. For $n=1$,
$$q\log r > d (\log r)^q,\;\; \text{for}\;\;r\;\;\text{large enough}.$$
Suppose that (\ref{gl01}) is true for some $n \in \mathbb{N}.$ Then
$$d (\log ^{n+1} r)^q= d (\log ^n (\log r))^q < \log^n ((\log r)^q),\;\;\text{for}\;\;r\;\;\text{large enough}.$$
Hence, in order to prove (\ref{gl01}) it suffices to show that there exists $R>0$ such that
$$\log^{n+1}(r^q)> \log ^n ((\log r)^q),\;\;\text{for}\;\;r \geq R,\;n \in \mathbb{N},$$
or equivalently that
$$\log (r^q) >(\log r)^q,\;\;\text{for}\;\;r \geq R,$$
which is true, and so, the result follows. 
\end{proof}
We now prove the following result.
\begin{theorem}
\label{folpomlog}
Let $f$ be a transcendental entire function of finite order and positive lower order. Then $f$ is $m$-log-regular.
\end{theorem}

\begin{proof}
Let $f$ be a transcendental entire function of finite order and positive lower order. We begin by noting that there exist $0<q<p$ such that
\begin{equation}
\label{ord}
e^{r^q} \leq M(r) \leq e^{r^p},\;\;\text{for}\;\;r\;\;\text{large enough}
\end{equation}

By the definition of $\mu_{m,\varepsilon},$ in order to prove that $f$ is $m$-log-regular, that is, that, for any $\varepsilon \in (0,1),$ there exist $R>0$ and $k>1$ such that $f$ satisfies (\ref{mlog}) or, equivalently,
\begin{equation}
\label{mlog1}
\varepsilon \log^m M(\exp^{m-1}(r^k))\geq \log^m( \exp^{m-1}(M(r)^k)),\;\;\text{for}\;\;r\;\;\text{large enough}.
\end{equation}
But (\ref{ord}) implies that 
$$\log^m M(\exp^{m-1}(r^k)) \geq \log^{m-1}((\exp^{m-1}(r^k))^q)$$
and 
$$\log^m( \exp^{m-1}(M(r)^k)) \leq \log^m( \exp^m(kr^p))=kr^p,$$
and so (\ref{mlog1}) is implied by
$$\varepsilon \log^{m-1}((\exp^{m-1}(r^k))^q)\geq kr^p,$$ that is,
\begin{equation}
\label{mlog2}
(\exp^{m-1}(r^k))^q \geq \exp^{m-1}(\frac{kr^p}{\varepsilon}),\;\;\text{for}\;\;r\;\;\text{large enough}.
\end{equation}
We set $r^k= \log^{m-1} s$ and (\ref{mlog2}) becomes
\begin{equation}
\label{mlog3}
\log^{m-1} (s^q) \geq \frac{k}{\varepsilon} (\log^{m-1} s)^{p/k},\;\;\text{for}\;\;s
\;\;\text{large enough}. 
\end{equation}
If we choose $k> p/q$ then, for any $\varepsilon \in (0,1)$, (\ref{mlog3}) holds for $s$ large enough, by Lemma \ref{genlemma2}.\end{proof}

 It is easy to see that if we combine Theorem \ref{mlogmweak} and Theorem \ref{folpomlog} we obtain Theorem \ref{ordqm}.

\section{Growth conditions for $Q_m(f)=A(f)$}
In this section we give a second proof of Theorem \ref{ordqm} by introducing a growth condition, given in the following theorem, which implies that $Q_m(f)=A(f).$ The same lower bound for $m=2$ appears in \cite[Theorem 6]{Smallg}.
\begin{theorem}
\label{ineqvreg}
Let $f$ be a transcendental entire function, $m \geq 2$ and $\phi_m(t)= \log^{m-1} M(\exp^{m-1}(t)).$ If there exist $0<q<1$ and $0<\tilde{q}< \infty$ such that, for some $n\geq 0$,
\begin{equation}
\label{dineqreg}
\exp^{n+m-1}((\log^{n+m-2}t)^q)\leq \phi_m(t) \leq \exp^{n+m-1}((\log^{n+m-2}t)^{\tilde{q}}),\;\;\text{for}\;\;t\;\;\text{large enough},
\end{equation}
then 
\begin{itemize}
\item[(i)]for any $d>1$ there exists $t_0>0$ such that 
$$ \phi_m(\psi_m(t)) \geq (\psi_m(\phi_m(t)))^d,\;\;\text{for}\;\;t \geq t_0,$$
where $\psi_m(t)= \exp^{n+m-1}((\log^{n+m-1}t)^p),$ $pq>1;$\\

\item[(ii)] $f$ is m-weakly regular and so $Q_m(f)=A(f).$
\end{itemize}
\end{theorem}
\begin{remark}
As $0<q<1$, the left bound in (\ref{dineqreg}) becomes smaller as $n$ increases. If we also take $\tilde{q}>1,$ then the right bound increases with $n$ and hence the condition (\ref{dineqreg}) is  more easily satisfied for larger $n$. As we will prove in Theorem \ref{order-growth}, all functions of positive lower order and finite order satisfy (\ref{dineqreg}).
\end{remark} 
\begin{proof}
(i) We have that 
\begin{eqnarray}
\psi_m(\phi_m(t)) &\leq& \exp^{n+m-1}((\log^{n+m-1}(\exp^{n+m-1}(\log^{n+m-2}t)^{\tilde{q}}))^p)\nonumber \\
&=& \exp^{n+m-1}((\log^{n+m-2}t)^{\tilde{q} p}) \nonumber
\end{eqnarray}
and also
\begin{eqnarray}
\phi_m(\psi_m(t)) &\geq& \exp^{n+m-1}((\log^{n+m-2}(\exp^{n+m-1}((\log ^{n+m-1}t)^p)))^q) \nonumber \\
&=& \exp^{n+m-1}((\exp ((\log^{n+m-1}t)^p))^q)\nonumber \\
&=&\exp^{n+m}(q(\log^{n+m-1}t)^p)\nonumber \\
&\geq & \exp^{n+m}((\log^{n+m-1}t)^{pq})\nonumber \\
& \geq & (\exp^{n+m-1} ((\log^{n+m-2}t)^{p \tilde{q}}))^d,\;\;\text{for any}\;\;d>1\;\;\text{and for}\;\;t\;\;\text{large enough},\nonumber 
\end{eqnarray}
since putting $w= \log^{n+m-1}t$ gives 
\begin{eqnarray}
\frac{\exp^{n+m-1}((\log^{n+m-1}t)^{pq})}{\exp^{n+m-2}((\log^{n+m-2}t)^{p \tilde{q}})}&=& \frac{\exp^{n+m-1}(w^{pq})}{\exp^{n+m-2}((e^w)^{p\tilde{q}})}\nonumber \\
&=& \frac{\exp^{n+m-1}(w^{pq})}{\exp^{n+m-2}(e^{p\tilde{q}w})}\nonumber \\
&=& \frac{\exp^{n+m-1}(w^{pq})}{\exp^{n+m-1}(p\tilde{q}w)} \to \infty\;\;\text{as}\;\;w \to \infty \nonumber
\end{eqnarray}
 and so $$\exp^{n+m-1}((\log^{n+m-1}t)^{pq}) \geq d\exp^{n+m-2}((\log^{n+m-2}t)^{p \tilde{q}}),\;\;\text{for}\;\;t\;\;\text{large enough}.$$
 Thus
$$ \phi_m(\psi_m(t)) \geq (\psi_m(\phi_m(t)))^d,\;\;d>1,\;\;\text{for}\;\;t\;\;\text{large enough}.$$
(ii) Now let $\phi_{m,{\varepsilon}}(t)= \phi_m(t)^{\varepsilon}$ and note that from the definition of $\phi_m$,
$$M^n(r)= \exp^{m-1}(\phi_m^n(\log^{m-1} r)).$$ 
Note also that 
\begin{eqnarray}
\mu_{m,{\varepsilon}}(r)&=& \exp^m(\varepsilon \log^m M(r)) \nonumber \\
&=& \exp^m(\varepsilon \log^m(\exp^{m-1}\phi_m(\log^{m-1}r))) \nonumber \\
&=& \exp^m(\log(\phi_{m,\varepsilon}(\log^{m-1}r))\nonumber \\
&=& \exp^{m-1}(\phi_{m,\varepsilon}(\log^{m-1}r)),\nonumber
\end{eqnarray}
and so, 
$$\mu_{m,{\varepsilon}}^n(r)= \exp^{m-1}(\phi_{m,{\varepsilon}}^n(\log^{m-1}r)).$$
Hence, in order to show that there exists $r= r(R)>0$  such that
$$\mu_{m,{\varepsilon}}^n(r) \geq M^n(R),\;\;\text{for}\;\;n \in \mathbb{N},$$
it suffices to show that there exists $r= r(R)>0$  such that
$$\phi_{m,{\varepsilon}}^n(r) \geq \phi_m^n(R),\;\;\text{for}\;\;n \in \mathbb{N}.$$
We showed in (i) that, for any $d>1$, $\phi_m(\psi_m(t)) \geq (\psi_m(\phi_m(t)))^d$, if $t$ is sufficiently large, or, equivalently, that given $\varepsilon>0$
$$\phi_{m,{\varepsilon}}(\psi_m(t)) \geq \psi_m(\phi_m(t)),\;\;\text{for}\;\;t\;\;\text{large enough}.$$
Therefore, 
$$\phi_{m,{\varepsilon}}(s) \geq \psi_m(\phi_m(\psi_m^{-1}(s))),\;\;\text{for}\;\;s\;\;\text{large enough}.$$
Since $\psi_m(t) \geq t$, by iterating we obtain
$$\phi_{m,{\varepsilon}}^n(s) \geq \psi_m(\phi_m^n(\psi_m^{-1}(s)))\geq \phi_m^n(\psi_m^{-1}(s)),\;\;\text{for}\;\;s\;\;\text{large enough}.$$
The result follows. 
\end{proof}

In order to complete the proof of Theorem \ref{ordqm} it remains to show that Theorem \ref{ineqvreg} can be applied to functions of finite order and positive lower order.
\begin{theorem}
\label{order-growth}
Let $f$ be a transcendental entire function of finite order and positive lower order. Then $f$ satisfies the hypotheses of Theorem \ref{ineqvreg} and hence $Q_m(f)=A(f).$
\end{theorem}
\begin{proof}
As $f$ is of finite order and positive lower order, (\ref{ord}) implies that, for $m \geq 2$ there exist $q \in (0,1)$ and $p \in (q, \infty)$ such that 
\begin{equation}
\label{ord-g}
\log^{m-2}((\exp^{m-1}t)^q) \leq \phi_m(t)= \log^{m-1} M(\exp^{m-1}t) \leq \log^{m-2}((\exp^{m-1}t)^p),
\end{equation}
for $t$ large enough.

In order to show that (\ref{ord-g}) implies (\ref{dineqreg}), it suffices to show that 
\begin{equation}
\label{og1}
\exp^{m-1}((\log^{m-2}t)^q) \leq \log^{m-2}((\exp^{m-1}t)^q)
\end{equation}
and
\begin{equation}
\label{og2}
\log^{m-2}((\exp^{m-1}t)^p) \leq \exp^{m-1}((\log^{m-2}t)^p),
\end{equation}
for $t$ large enough.
Note that (\ref{og1}) is equivalent to 
$$(\log^{m-2} t)^q \leq \log ^{2m-3} (\exp ^{m-1} t)^q,\;\;\text{for}\;\;t\;\;\text{large enough},$$
which, for $s= \exp^{m-1} t,$ becomes
\begin{equation}
\label{applygl}
(\log^{2m-3}s)^q \leq \log^{2m-3} s^q\;\;\text{for}\;\;s\;\;\text{large enough}.
\end{equation}
By Lemma \ref{genlemma2}, (\ref{applygl}) holds for $s$ large enough and hence so does (\ref{og1}).

Similarly, using Lemma \ref{genlemma2}, one can show that (\ref{og2}) is true.\\
Therefore, the hypotheses of Theorem \ref{ineqvreg} are satisfied for $q$ and $p=\tilde{q}.$ \end{proof}

\section{2-log-regularity and strong log-regularity}
In \cite{Evd16} we introduced a sufficient condition for $Q_2(f)=A(f)$ called strong log-regularity. A transcendental entire function $f$ is strongly log-regular if,  for any $\varepsilon \in (0,1),$ there exist $R>0$ and $k>1$ such that, for $r>R,$ 
\begin{equation}
\label{nrc}
\log M(r^k) \geq (k\log M(r))^{1/\varepsilon}.
\end{equation} 
Both strong log-regularity and $2$-log-regularity imply $Q_2(f)=A(f)$ and also any transcendental entire function of finite order and positive lower order is both strongly log-regular and $2$-log-regular. Therefore it is of interest to know how these two conditions are related. For a function of finite order we have the following result.
\begin{theorem}
\label{strong2-log}
Let $f$ be a transcendental entire function of finite order. If $f$ is strongly log-regular then $f$ is $2$-log-regular.
\end{theorem}
\begin{proof}
As $f$ is of finite order, (\ref{ord}) implies that
there exists $p \geq 0,$ such that
\begin{equation}
\label{logord}
\log M(r) \leq r^p,\;\;\text{for}\;\;r\;\;\text{large enough}.
\end{equation}
Also since $f$ is strongly log-regular, for any $\varepsilon \in (0,1),$ there exist $R>0$ and $k>1$ such that
\begin{equation}
\label{st-m}
\log M(r^k) \geq (k\log M(r))^{1/\varepsilon},\;\;\text{for}\;\;r > R.
\end{equation}
In order to show that $f$ is $2$-log-regular we will show that, for any $\varepsilon \in (0,1),$ 
$$\mu_{2,\varepsilon}(\exp(r^k)) \geq \exp(M(r)^k),\;\;\text{for}\;\;r\;\;\text{large enough};$$
that is, using the definition of $\mu_{2,\varepsilon}(r),$
\begin{equation}
\label{5.1.1}
(\log M(\exp(r^k)))^{\varepsilon} \geq M(r)^k,\;\;\text{for}\;\;r\;\;\text{large enough}.
\end{equation}
It is obvious from the definition of $2$-log-regularity that if the condition holds for any $\varepsilon \in (0, 1/e^p)$ it will hold for any $\varepsilon \in (0,1)$ and so we now fix $\varepsilon \in (0, 1/e^p)$ and show that (\ref{5.1.1}) holds for this value of $\varepsilon.$

Consider now 
\begin{equation}
\label{ndef}
n= \frac{k\log r -\log \log r}{\log k}.
\end{equation}
Then  $k^n= r^k/\log r,$ which gives us that
$\exp(r^k)= r^{k^n}.$
Hence 
$$(\log M(\exp(r^k)))^{\varepsilon}= (\log M(r^{k^n}))^{\varepsilon},$$
and by applying (\ref{st-m}) $n$ times, we deduce that
$$ (\log M(\exp(r^k)))^{\varepsilon} \geq k^{1+ 1/\varepsilon+...+1/{\varepsilon}^{n-1}} (\log M(r))^{1/\varepsilon^{n-1}},\;\;\text{for}\;\;r\;\;\text{large enough}.$$
Therefore, it suffices to show that
$$(\log M(r))^{1/\varepsilon^{n-1}} \geq  M(r)^k,$$
or, equivalently, that
$$\left(\frac{1}{\varepsilon}\right)^{n-1} \log \log M(r) \geq k\log M(r),\;\;\text{for}\;\;r\;\;\text{large enough}.$$
By (\ref{logord}) it is sufficient to show that 
$$\left(\frac{1}{\varepsilon}\right)^{n-1}  \geq kr^p,$$
or, equivalently,
\begin{equation}
\label{n-1}
(n-1) \log \frac{1}{\varepsilon} \geq \log k + p\log r,\;\;\text{for}\;\;r\;\;\text{large enough}.
\end{equation}
In order to show that (\ref{n-1}) is true we first note that it follows from (\ref{ndef}) that
$$n-1= \frac{k\log r}{\log k}- \frac{\log \log r + \log k}{\log k},$$  and so
\begin{equation}
\label{explanation}
(n-1) \log \frac{1}{\varepsilon} -p \log r= \left(\log \frac{1}{\varepsilon}\frac{k}{\log k}-p\right)\log r - \log \frac{1}{\varepsilon}\frac{\log \log r + \log k}{\log k}.
\end{equation}
Since $\log \frac{1}{\varepsilon} >p$, there exists $R_0>0$ such that
$$\left(\log \frac{1}{\varepsilon}\frac{k}{\log k}-p\right)\log r \geq \log \frac{1}{\varepsilon}\frac{\log \log r + \log k}{\log k}+  \log k,\;\;\text{for}\;\;r \geq R_0.$$
Together with (\ref{explanation}), this is sufficient to prove (\ref{n-1}).\end{proof}

The converse of Theorem \ref{strong2-log} is not always true though. We now use a function, that was constructed by Rippon and Stallard in \cite[Example 6.1]{Regul}, in order to prove that there exists a $2$-log-regular function of finite order which is not strongly log-regular.

We will need the following result:
\begin{lemma}
\label{sim/log}
Let $\phi$ and $\psi$ be real functions defined on $(0, \infty)$ with $\liminf_{t \to \infty} \phi(t) >1$, $\liminf_{t \to \infty} \psi(t) >1$  and such that 

$$\phi(t) \sim \psi(t),\;\;\text{as}\;\;t \to \infty.$$
Then, for any $\varepsilon \in (0,1),$ there exist $t_0>0$ and $k>1$ such that
\begin{equation}
\label{real2log}
\phi(e^{kt}) \geq \exp(\frac{k}{\varepsilon} \phi(t)),\;\;\text{for}\;\;t\geq t_0
\end{equation}
if and only if there exist $t_1>0$ and $k'>1$ such that
$$\psi(e^{k't}) \geq \exp(\frac{k'}{\varepsilon} \psi(t)),\;\;\text{for}\;\;t\geq t_1.$$

\end{lemma}
\begin{proof}
Let 
\begin{equation}
\label{swr9}
\phi(t) = \psi(t) (1+\epsilon(t)),
\end{equation}
 where $\epsilon(t) \to 0$ as $t \to \infty$ and suppose that $\phi$ satisfies (\ref{real2log}). Then
\begin{equation}
\label{swr10}
\log \phi(t^k) \geq \frac{k}{\varepsilon}\phi(\log t),\;\;\text{for}\;\; \log t\geq t_0.
\end{equation}
It follows from (\ref{swr9}) and (\ref{swr10}) that
$$\log \psi(t^k)+ \log (1+ \epsilon(t^k)) \geq \frac{k}{\varepsilon} \psi(\log t)(1+ \epsilon(\log t)),\;\;\text{for}\;\; \log t\geq t_0.$$
Hence, since $1+\epsilon(\log t) \to 1$ as $t\to \infty,$ $\log (1+ \epsilon(t^k)) \to 0$, as $t \to \infty,$ and $\liminf_{t \to \infty} \phi(t) >1$, $\liminf_{t \to \infty} \psi(t) >1$, there exist $k'>1$ and  $t_1>0,$ such that
$$\log (\psi(t^{k'})) \geq \frac{k'}{\varepsilon} \psi(\log t),\;\;\text{for}\;\;t\geq t_1,$$
as claimed.\end{proof}

\begin{ex}
\label{2notstrong}
There exists a transcendental entire function of finite order which is $2$-log-regular but not strongly log-regular.
\end{ex} 
\begin{proof}
The main idea of the proof is to use a function $f$ constructed by Rippon and Stallard \cite[Example 6.1]{Regul}, which has order zero and is not log-regular (and hence is not strongly log-regular) and show that $f$ is $2$-log-regular.\\

In order to show that $f$ is $2$-log-regular we need to show that, for any $\varepsilon>0,$ there exist $r_0>0$ and $k>1$ such that
$$\mu_{2,\varepsilon}(\exp(r^k)) \geq \exp (M(r)^k),\;\;\text{for}\;\;r \geq r_0,$$
or equivalently,
$$\log (M(\exp(r^k)) \geq M(r)^{k/\varepsilon},\;\;\text{for}\;\;r \geq r_0.$$
Hence, the condition we have to prove for $\psi (t) = \log M(e^t)$ is that, for any $\varepsilon \in (0,1),$ there exist
$t_0>0$ and $k>1$ such that
\begin{equation}
\label{ex3}
\psi(\exp(kt)) \geq \exp (\frac{k}{\varepsilon} \psi(t)), \;\;\text{for}\;\;t \geq t_0.
\end{equation} 

In Rippon and Stallard's example, $\phi(t)= (\log M(e^t))/(1+ \epsilon(t))$ was defined as follows:\\
$$\phi(t)= \begin{cases} \mu_n(t),\;\;t \in [t_{n+1}^{3/4}, t_{n+1}],\\
\mu(t),\;\;\text{otherwise}, \end{cases}$$
where $\mu(t)= \exp(t^{1/2})$ and $\mu_n(t)$ denotes the linear function such that $\mu_n(t)=\mu(t)$ for $t= t_{n+1}^{3/4}, t= t_{n+1}.$\\

We will first show that for any $\varepsilon \in (0,1),$ there exist
$t_1>0$ and $k'>1$ such that
\begin{equation}
\label{example}
\phi(\exp(k't)) \geq \exp (\frac{k'}{\varepsilon} \phi(t)), \;\;\text{for}\;\;t \geq t_1.
\end{equation}\\

Let $\varepsilon \in (0,1)$.  When $\phi(t)= \mu(t) = \exp(t^{1/2}),$ we have
$$\phi(\exp(kt)) \geq \mu(\exp(kt))= \exp(\exp(\frac{1}{2}kt)) \geq \exp(\frac{k}{\varepsilon}\exp(t^{1/2}))=\exp(\frac{k}{\varepsilon}\phi(t)),$$ for $t$ large enough,
and so (\ref{example}) holds, for these values of $t$.

Now suppose that $t \in [t_{n+1}^{3/4}, t_{n+1}]$, for some $n \in \mathbb{N}.$ Then
\begin{eqnarray}
\phi(\exp(kt))\geq \phi(\exp(kt_{n+1}^{3/4})) &\geq& \mu(\exp(kt_{n+1}^{3/4})) \nonumber \\
&=& \exp(\exp(\tfrac{1}{2}kt_{n+1}^{3/4})) \nonumber\\
&\geq & \exp(\tfrac{k}{\varepsilon} \exp(t_{n+1}^{1/2})),\;\;\text{for}\;\;t_{n+1}\;\;\text{large enough}, \nonumber \\
&=& \exp(\tfrac{k}{\varepsilon}\phi(t_{n+1}))\nonumber\\
&\geq & \exp (\tfrac{k}{\varepsilon} \phi(t)),\nonumber
\end{eqnarray}
and hence (\ref{example}) is satisfied.

Now, Lemma \ref{sim/log} implies that $\psi$ satisfies (\ref{ex3}) which means that $f$ is $2$-log-regular.\end{proof}

 \end{document}